\UseRawInputEncoding
\documentclass[a4paper,dvips,10pt,oneside]{article}
\usepackage[makeroom]{cancel}
\usepackage{graphicx}
\usepackage{epstopdf}
\usepackage{subfig}
\usepackage{color}
\usepackage{float}
\usepackage{leftidx}
\usepackage{bigints}
\usepackage[affil-it]{authblk}
\usepackage{amssymb, amsmath,dsfont,url,epsfig}
\usepackage[left=1in, right=1in, top=1.1in, bottom=1.1in, includefoot, headheight=13.6pt]{geometry}
\usepackage[T1]{fontenc}
\usepackage{titlesec, blindtext, color}
\definecolor{gray75}{gray}{0.75}
\providecommand{\keywords}[1]{\textbf{keywords:}#1}

\newcommand{\sln}{\linespread{1}}
\newcommand*{\email}[1]{\href{mailto:#1}{\nolinkurl{#1}} } 
\titleformat{\chapter}[block]{\LARGE\bfseries\sln}{Chapter \thechapter}{11pt}{\newline\huge\bfseries}
\newtheorem{theorem}{Theorem}[section]
\newtheorem{remark}{Remark}[section]
\newtheorem{definition}{Definition}[section]

\newenvironment{proof}{\paragraph{Proof:}}{\hfill$\square$}
\newtheorem{lemma}{Lemma}[section]
\newtheorem{proposition}{Proposition}[section]
\newtheorem{corollary}{Corollary}[section]
\begin{document}
\title{The Isoperimetric Problem In Randers Planes}

\author{Arti Sahu
  \thanks{E-mail: \texttt{arti.sahu@bhu.ac.in}}}
\affil{DST-CIMS, Banaras Hindu University, Varanasi-221005, India}

\author{Ranadip Gangopadhyay
  \thanks{E-mail: \texttt{ranadip.gangopadhyay1@bhu.ac.in}}}
\affil{DST-CIMS, Banaras Hindu University, Varanasi-221005, India}

\author{Hemangi Madhusudan Shah
  \thanks{E-mail: \texttt{hemangimshah@hri.res.in}}}
\affil{Harish-Chandra Research Institute, A CI of Homi Bhaba National Institute, Chhatnag Road, Jhunsi, Prayagraj-211019, India}

\author{Bankteshwar Tiwari
  \thanks{E-mail: \texttt{btiwari@bhu.ac.in}}}
\affil{DST-CIMS, Banaras Hindu University, Varanasi-221005, India}

\maketitle
\begin{abstract}
In this paper, the isoperimetric problem in Randers planes, $(\mathbb{R}^2,F=\alpha +\beta)$, which are slight deformation of the Euclidean plane $(\mathbb{R}^2,\alpha)$ by suitable one forms $\beta$, have been studied. We prove that the circles centred at the origin achieves the local maximum area of the isoperimetric problem with respect to well known volume forms in Finsler geometry.
\end{abstract}
\hspace{1.0cm}\keywords{~Isoperimetric problem, Randers Planes, Minkowski Planes, Calculus of variations.}
\section{Introduction}
The isoperimetric problem has a history of more than two thousand years. The original isoperimetric problem can be posed as: to find a simple closed curve of given perimeter which encloses the maximum area. Mathematically, the isoperimetric problem is to find a simple closed curve for which the area achieves the equality in the isoperimetric inequality $L^2\ge 4\pi A$.

The ancient greeks knew that a circle has greater area than any polygon with the same perimeter \cite{VB}. The famous Roman poet Virgil mentioned about this problem in his epic `Aeneid'. Although the statement is really very simple and the possible answer was known to the mankind, but the first step to obtain the mathematical proof was due to Swiss geometer J. Steiner in $1839$. He proved that circle is the only possible solution in plane geometry. The first complete proof was given by Weierstrass  using the techniques of calculus of variations. After that many different proofs were published by several mathematicians including Hurwitz, Schimdt \cite{AH, ES, JS}. For constant curvature spaces the isoperimetric inequality is given by $L^2\ge 4\pi A-kA^2,$
where $k$ is the curvature of the space. The study of the isoperimetric inequality has been done extensively in e.g. \cite{AGBE,MG,MP,STY}. For the Finsler case  Zhou studied the isoperimetric problem in 2-dimensional spherically symmetric spaces with zero curvature \cite{LZ1} with respect to the Busemann-Hausdorff volume form. Zhan and Zhou studied this problem with respect to the Holmes-Thompson volume form \cite{LZ2}. More precisely, they studied the isoperimetric problem for the following Berwald metric,
\begin{equation*}
F(x,y)=\frac{(\sqrt{|y|^2-(|x|^2|y|^2)-\langle x,y\rangle^2}+\langle x,y\rangle)^2}{(1-|x|^2)(\sqrt{|y|^2-(|x|^2|y|^2)-\langle x,y\rangle^2)}}
\end{equation*} 
and showed that the circles $\gamma_0(t)=(a\cos t, a\sin t)$ are the local solutions of the isoperimetric problem in $2$-dimensional spherically symmetric Finsler metrics with zero flag curvature. Li and Mo \cite{LM} studied the isoperimetric problem for the two dimensional Funk space, which is a spherically symmetric metric with constant flag curvature $-1$ and is given by,
\begin{equation*}
F(x,y)=\frac{(\sqrt{|y|^2-(|x|^2|y|^2)-\langle x,y\rangle^2}+\langle x,y\rangle)}{2(1-|x|^2)}.
\end{equation*}
It should be noted that the Funk metric, $F$ defined above, is an example of Randers metric. In this paper, we consider the Randers metric, which is a class of simplest non-Riemannian Finsler metric introduced by G. Randers, and is defined by $F=\alpha +\beta$, where $\alpha$ is a Riemannian metric and $\beta$ is a one-form with $b:=\|\beta\|_{\alpha}<1$, \cite{GR}.\\
In order to simplify the calculations we choose $\beta= b(\cos \theta(x^1,x^2) dx^1 +\sin\theta(x^1,x^2) dx^2)$, where $\theta$ is a function of $x=(x^1,x^2)\in \mathbb{R}^2$ and $0<b<1$. We show that under certain conditions on $\theta$, circles centred at the origin are solutions of the isoperimetric problem locally with respect to different known volume forms on the Randers manifolds. More precisely, we prove the following:
\begin{theorem}\label{thm1}
Let $(\mathbb{R}^2, F)$, be a Randers plane, where $$F=\sqrt{(dx^1)^2+(dx^2)^2}+ b(\cos \theta(x^1,x^2) dx^1 +\sin\theta(x^1,x^2) dx^2).$$ If $\theta=c$, or,  $t+c$, for some constant $c$ on the circles of the form $\gamma_0=(a\cos t,a\sin t)$, then $\gamma_0$ is a solution of the isoperimetric problem, with respect to each of the volume forms: the Busemann-Hausdorff, the Holmes-Thompson, the Maximum, and the Minimum.
\end{theorem}
If we choose $\theta=0$, then $F=\sqrt{(dx^1)^2+(dx^2)^2}+ bdx^1$ with $0<b<1$. In this case the Finsler metric $F$ defined on  $\mathbb{R}^2$ gives a Minkowski space as $F$ is independent of the point on $\mathbb{R}^2$ and therefore, $(\mathbb{R}^2, F)$ has zero flag curvature. For this Randers-Minkowski metric, we obtain:
\begin{corollary}\label{cor1}
Let $(\mathbb{R}^2, F)$ be a Randers plane, where $F=\sqrt{(dx^1)^2+(dx^2)^2}+ bdx^1$, $0<b<1$ . Then $\gamma_0(t)=(a\cos t,a\sin t)$ is a solution of the Minkowski isoperimetric problem.
\end{corollary}
Considering $\theta \equiv constant$ in Theorem \ref{thm1}, we obtain:
\begin{corollary}\label{cor2}
Consider the Finsler metric $F=\sqrt{(dx^1)^2+(dx^2)^2}+d\tau$, where $\tau$ is a smooth real valued function on $\mathbb{R}^2$ with $\|d\tau\|$ constant and $\|d\tau\|<1$ ($\|.\|$ denotes the Euclidean norm on $\mathbb{R}^2$). Then $\gamma_0(t)=(a\cos t,a\sin t)$ is a solution of the isoperimetric problem.
\end{corollary}
\begin{remark}
 As per Busemann \cite{HB}, circle may not be a solution of the isoperimetric problem in Minkowski plane. However, for the Minkowski plane discussed above, indeed the circle $\gamma_0$ turns out to be a solution of the isoperimetric problem.
\end{remark}
\section{Preliminaries}
Let $ M $ be an $n$-dimensional smooth manifold. $T_{x}M$ denotes the tangent space of $M$
 at $x$. The tangent bundle of $ M $ is the disjoint union of tangent spaces, $ TM:= \sqcup _{x \in M}T_xM $. We denote the elements of $TM$ by $(x,y)$ where $y\in T_{x}M $ and $TM_0:=TM \setminus\left\lbrace 0\right\rbrace $, the slit tangent bundle of $M$.
 \begin{definition}
 \cite{SSZ} A Finsler metric on $M$ is a function $F:TM \to [0,\infty)$ satisfying the following conditions:
 \begin{itemize}
 \item[(i)] $F$ is smooth on $TM_{0}$,
  \item[(ii)] $F$ is positively 1-homogeneous on the fibers of the tangent bundle $TM$,
   \item[(iii)] The Hessian of $\frac{F^2}{2}$ with element $g_{ij}=\frac{1}{2}\frac{\partial ^2F^2}{\partial y^i \partial y^j}$ is positive definite on $TM_0$.
 \end{itemize}
 The pair $(M,F)$ is called a Finsler space and $g_{ij}$ is called the fundamental tensor.
 \end{definition}
 It is known that there is a canonical volume form in a Riemannian manifold $(M^n,\alpha)$, $\alpha=\sqrt{a_{ij}dx^idx^j}$, given by $dV=\sqrt{\det(a_{ij})}dx$, whereas in Finsler manifold there are several volume forms some of them are defined as follows: 
 \begin{definition}
  For an $n$-dimensional Finsler manifold $(M^n,F)$, the Busemann-Hausdorff volume form is defined as $dV_{BH}=\sigma_{BH}(x)dx$, where
     \begin{equation}\label{1.01}
     \sigma_{BH}(x)=\frac{vol(B^n(1))}{vol\left\lbrace (y^i)\in T_xM : F(x,y)< 1 \right\rbrace }.
     \end{equation}
   The Holmes-Thompson volume form is defined as $dV_{HT}=\sigma_{HT}(x)~dx$, where 
     \begin{equation}\label{1.02}
     \sigma_{HT}(x)=\frac{1}{vol(B^n(1))}\int_{F(x,y)<1}\det(g_{ij}(x,y))dy.
     \end{equation}
     Here, $B^n(1)$ is the Euclidean unit ball in $\mathbb{R}^n$ and $vol$ is the Euclidean volume.\\
     \end{definition}
     \begin{definition}
    The maximum volume form of a Finsler metric $F$ with fundamental metric tensor $g_{ij}$ is defined as
 \begin{equation}
 dV_{\max}=\sigma_{\max}(x)dx,
 \end{equation}
 where $\sigma_{\max}(x)=\max\limits_{y \in I_x}\sqrt{\det(g_{ij}(x,y))}$ and, $I_x=\left\lbrace y\in T_xM  |  F(x,y)=1\right\rbrace$, is the indicatrix at the point $x$ of the Finsler manifold. \\
  The minimum volume form of a Finsler metric $F$ with the fundamental metric tensor $g_{ij}$ is defined as
  \begin{equation}
  dV_{\min}=\sigma_{\min}(x)dx,
  \end{equation}
  where $\sigma_{\min}(x)=\min\limits_{y\in I_x}\sqrt{\det(g_{ij}(x,y))}$.
 \end{definition} 
 \begin{proposition}\label{prop1}\cite{XZ}
  Let $F=\alpha\phi(s)$, $s=\beta/\alpha$ be an $(\alpha,\beta)$-metric on an $n$-dimensional manifold $M$ and $T=\phi(\phi-s\phi')^{n-2}\left[ \phi-s\phi'+(b^2-s^2)\phi^{''}\right] $. Then the volume form $dV$ of the Finsler metric $F$ is given by 
  $$dV=f(b)dV_{\alpha},$$
  where
  \begin{equation}\label{eqn3.1}
  f(b) := \begin{cases} \frac{\int\limits_{0}^{\pi}\sin^{n-2}(t)dt}{\int\limits_{0}^{\pi}\frac{\sin^{n-2}(t)}{[\phi(b\cos (t))]^n}dt}, &\mbox{if } dV=dV_{BH}, \\
  \frac{\int\limits_{0}^{\pi}\sin^{n-2}(t)T(b\cos (t))dt}{\int\limits_{0}^{\pi}\sin^{n-2}(t)dt}, & \mbox{if } dV=dV_{HT}, \end{cases}
  \end{equation}
 and $dV_{\alpha}=\sqrt{\det(a_{ij})}dx$ denotes the Riemannian volume form $\alpha$.
  \end{proposition}
    \begin{lemma}\label{lem1.1}\cite{SSZ}
    The Busemann-Hausdorff volume form of a Randers metric $F=\alpha+\beta$ is given by,
         \begin{equation}
         dV_{BH}=(1-\|\beta\|^2_{\alpha})^{\frac{n+1}{2}}dV_{\alpha}
         \end{equation}
          and the Holmes-Thompson volume form is given by,
         
         \begin{equation}
         dV_{HT}=dV_{\alpha}.
         \end{equation}
    \end{lemma}
  \begin{lemma}\cite{BYW}
  The maximum volume form of a Randers metric $F=\alpha+\beta$ is given by
       \begin{equation}
       dV_{\max}=(1+\|\beta\|_{\alpha})^{n+1}dV_{\alpha}
       \end{equation} and the minimum volume form is given by
       
       \begin{equation}
       dV_{\min}=(1-\|\beta\|_{\alpha})^{n+1}dV_{\alpha}.
       \end{equation}
  \end{lemma}
 Now we discuss about the solvability of the isoperimetric problem. First we recall some required definitions. The more details can be found in \cite{MH, LZ1}.
\begin{definition}
 A continuous function $ \gamma : \left[t_0,t_1 \right]  \to M^2 $ is called an \textit{admissible curve} on $M$, if there exists a partition $ P := \left\lbrace t_0=a_{0}<a_{1}<...<a_{k}=t_1\right\rbrace$ of the interval $[t_0,t_1]$ such that the curve $\alpha $ has continuous derivative in each subinterval $\left[ a_{i},a_{i+1} \right]$,~ for all $i = 0,1,...,k-1 $.
\end{definition}
 Consider the set of all admissible curves joining two fixed points for which the definite integral $L = \int\limits_{t_0}^{t_1}g(x^{1},x^{2},\dot{x}^{1},\dot{x}^{2})dt$ takes a constant value $l$. Then the isoperimetric problem is about finding an admissible curve $\gamma_0(t)=(x^1(t),x^2(t))$ for which the integral $A=\int\limits_{t_0}^{t_1}f(x^{1},x^{2},\dot{x}^{1},\dot{x}^{2})dt$ has its \textit{extremum}.\\
To solve the isoperimetric problem we first consider the Lagrange functional:
 \begin{equation}\label{eq1.01}
 J(x)=A+\lambda L=\int_{t_0}^{t_1}h(x^1,x^2, \dot{x}^1,\dot{x}^2,\lambda)dt,
 \end{equation}
where, 
\begin{equation}\label{eq1.02}
h=f+\lambda g.
\end{equation}
\begin{definition}\label{def1.001}
 An admissible curve $\gamma_0$ is said be an \textit{isoperimetric extremal}, if it satisfies the following Euler-Lagrange equations:
\begin{equation}\label{eq1.0.1}
\frac{\partial h}{\partial x^i}- \frac{d}{dt}\left(\frac{\partial h}{\partial \dot{x}^i} \right) =0, ~ i=1,2.
\end{equation}
And an admissible curve $\gamma_0$ satisfying the Euler-Lagrange equations is said to be \textit{normal} if $P_1$ and $P_2$ are non-zero functions, where $P_i=g_{x^i}-\frac{d}{dt}g_{\dot{x}^i}$, for $i=1,2$.
\end{definition}
The \textit{increment} of the functional $J$ is defined by 
\begin{equation}\label{1.24}
\bigtriangleup J=\int\limits_{t_{0}}^{t_{1}}E(x^{1},x^{2},\dot{x}^1,\dot{x}^{2},u^{1},u^{2})dt,
\end{equation}
where
\begin{equation}\label{1.25}
E(x^{1},x^{2},\dot{x}^1,\dot{x}^{2},u^{1},u^{2})=h(x^{1},x^{2},u^{1},u^{2})-h(x^{1},x^{2},\dot{x}^1,\dot{x}^{2})-\sum\limits_{j=1}^{2}(u^{j}-\dot{x}^{j})\frac{\partial h(x^{1},x^{2},\dot{x}^1,\dot{x}^{2})}{\partial \dot{x}^j}.
\end{equation}
The function $E$ is called the \textit{Weierstrass function} of $J$.\\
The \textit{second variation} of $J$ along the curve $\gamma_0$ is defined as
\begin{equation}
J''(\gamma_0,y)=\int\limits_{t_0}^{t_1}2\omega(t,y(t),\dot{y}(t))dt,
\end{equation}
where, $\omega =\sum\limits_{i,j=1}^{2} h_{x^ix^j}y^iy^j+2h_{x^i\dot{x}^j}y^i\dot{y}^j+h_{\dot{x}^i\dot{x}^j}\dot{y}^i\dot{y}^j$.
\begin{definition}
A point $t=c$ on $t_0 <t \le t_1$ is said to be a \textit{conjugate point} to $t=t_0$ on $\gamma_0$ if, there exist a solution $y_i(t)$ of the Jacobi equations
\begin{equation*}
\omega_{y^i}-\frac{d\omega_{\dot{y}^i}}{dt}=0, \quad i=1,2,
\end{equation*}
such that $\int\limits_{t_0}^{t_1}\left( (g)_{x^i}y^i + (g)_{\dot{x}^i}\dot{y}^i\right) dt = 0$, holds.
\end{definition}
A sufficiency theorem for a strong maximum of the isoperimetric problem  proved by Hestenes in \cite{MH} can be stated as:

\begin{theorem}\label{thm1.01}
Let $\gamma_0$ be an admissible curve. Suppose there exist $\lambda_0$ such that, relative to the function
$J(x) = \int\limits_{t_0}^{t_1} h(x^1, x^2, \dot{x}^1,\dot{x}^2,\lambda)dt$ the following holds:\\
(1) $\gamma_0$ is isoperimetric extremal, i.e., it satisfies the Euler-Lagrange equations \ref{eq1.0.1},\\
(2) $\gamma_0$ is normal,\\
(3) the Weierstrass function $E(x, \dot{x} , u) < 0$ for every admissible set $(x, u)$ with $u \ne k\dot{x} (k > 0)$,\\
(4) $J′′(\gamma_0, y)< 0$, for every admissible variations $y^i(t)\ne \rho(t)\dot{x}^i(t), (t_0 \le t \le t_1)$, vanishing at $t = t_0$ and $t = t_1$ and satisfying with $\gamma_0$ the equations $\int\limits_{t_0}^{t_1}(g_{x^i}y^i+g_{\dot{x}^i}\dot{y}^i)dt= 0 \quad i = 1, 2$, \\
(5) along $\gamma_0$ the inequality $\sum\limits_{i,j=1}^{2}h_{\dot{x}^i\dot{x}^j}y^iy^j<0$, holds for all $y\ne k\dot{\gamma_0}(t)$.\\
 Then $\gamma_0$ is a proper strong maximum of the isoperimetric problem.
\end{theorem}

 \section{Isoperimetric Problem  with respect to the well known volume forms}
 In this Section, we prove Theorem \ref{thm1} as stated in the introduction. We begin by proving the following Lemma.
\begin{proposition}\label{prop1.01}
      Let $F(x,y) = \alpha+\beta$ be the Randers metric, where $\alpha$ is the Euclidean metric and $\beta=b(\cos \theta(x^1,x^2) dx^1 + \sin \theta(x^1,x^2) dx^2)$ is a one form.  Let $\gamma(t) = (x^{1}(t),x^{2}(t))$ be a simple closed curve in $\mathbb{R}^2$, with $t\in[t_{0},t_{1}]$ and $\gamma(t_{0})=\gamma(t_{1})$. Then\\
      (i) The  Randers length $L$ of the curve  $\gamma$, is given by,
      \begin{equation}\label{1.1.05}
       L=\int\limits_{t_{0}}^{t_{1}}\left( \sqrt{({\dot{x}^{1}})^2+{(\dot{x}^{2}})^2}+b\left( \cos \theta(x^1,x^2) \dot{x}^{1} + \sin \theta(x^1,x^2) \dot{x}^{2}\right) \right) dt.
      \end{equation}
      (ii) The Busemann-Hausdorff area $A_{BH}$ enclosed by the curve $\gamma$ is given by,
       \begin{equation}\label{1.1.06}
      A_{BH} =\frac{(1-b^2)^{\frac{3}{2}}}{2}\int\limits_{t_{0}}^{t_{1}}(x^{1}\dot{x}^{2}-x^{2}\dot{x}^{1})dt.
      \end{equation}
      (iii) The Holmes-Thompson area $A_{HT}$ enclosed by the curve $\gamma$ is given by,
             \begin{equation}\label{1.1.006}
            A_{HT} =\frac{1}{2}\int\limits_{t_{0}}^{t_{1}}(x^{1}\dot{x}^{2}-x^{2}\dot{x}^{1})dt.
            \end{equation}
(iv) The maximum area $A_{max}$ with respect to the maximum volume form $dV_{max}$ enclosed by the curve $\gamma$ is given by,
             \begin{equation}\label{1.1.0006}
            A_{max} =\frac{(1+b)^3}{2}\int\limits_{t_{0}}^{t_{1}}(x^{1}\dot{x}^{2}-x^{2}\dot{x}^{1})dt.
            \end{equation}
(v) The minimum area $A_{min}$ with respect to the minimum volume form $dV_{min}$ enclosed by a simple closed curve $\gamma$ is given by,
             \begin{equation}\label{1.1.1006}
            A_{min} =\frac{(1-b)^3}{2}\int\limits_{t_{0}}^{t_{1}}(x^{1}\dot{x}^{2}-x^{2}\dot{x}^{1})dt.
            \end{equation}
      \end{proposition}
      \begin{proof}
$(i)$ The proof of $(i)$  is immediate.\\
$(ii)$ From Lemma \ref{lem1.1} the Busemann-Hausdorff volume form of a Randers metric $\alpha+\beta$ is given by $(1-b^2)^{\frac{n+1}{2}}dx$. Since, in our case $n=2$, we have Busemann-Hausdorff volume form of $F$ is $(1-b^2)^{\frac{3}{2}}dx^1dx^2$. Hence the area enclosed by the simple closed curve $\gamma$ is
 \begin{equation}
 A_{BH}=\iint\limits_{R}(1-b^2)^{\frac{3}{2}}dx^1dx^2,
 \end{equation}
 where $R$ is the region enclosed by the curve $\gamma$.
 Using Green's Theorem we obtain \ref{1.1.06} immediately.\\
 The proof of $(iii)$, $(iv)$, $(v)$ are analogous to $(ii)$. 
      \end{proof}\\

We first solve the isoperimetric problem for the Busemann-Hausdorff volume form. Let us consider the functional given by \ref{eq1.01}, where $L$ and $A_{BH}$ are given by \ref{1.1.05} and \ref{1.1.06}, respectively. Therefore, we obtain the Lagrange function of $J$ as:
 \begin{equation}\label{1.07}
h= f+\lambda g,
\end{equation}
where,
\begin{equation}
f=\frac{(1-b^2)^{\frac{3}{2}}}{2}(x^{1}\dot{x}^{2}-x^{2}\dot{x}^{1}), \quad g=\sqrt{(\dot{x}^{1})^2+(\dot{x}^{2})^2}+b(\cos \theta(x^1,x^2) \dot{x}^1 +\sin \theta(x^1,x^2) \dot{x}^2).
\end{equation}
 Let us consider the smooth curve,
 \begin{equation}\label{1.08}
\gamma(t) = (x^{1}(t),x^{2}(t)) = (r(t)\cos t, r(t)\sin t).
\end{equation}
Differentiate $\gamma(t)$ with respect to $t$ we get,
\begin{equation}\label{1.09}
\dot{\gamma}(t) = (\dot{x}^{1}(t),\dot{x}^{2}(t)) =(\dot{r}(t)\cos t-r(t)\sin t, \dot{r}(t)\sin t+r(t)\cos t).
\end{equation}
Hence, along $\gamma(t)$  we get,
 \begin{equation}\label{1.10}
x^{1}\dot{x}^2-x^{2}\dot{x}^1 = |x|^2 = r^2,~~ |\dot{x}|^{2} = r^{2}+\dot{r}^{2}.
\end{equation}
 Using \ref{1.10} in \ref{1.07} we get,
\begin{equation}\label{1.11}
h(r,\dot{r}, t) = \frac{1}{2}(1-b^2)^\frac{3}{2}r^2+ \lambda (\sqrt{r^2+\dot{r}^2}+b(\dot{r}\cos (\theta-t) + r\sin (\theta-t))).
\end{equation}
The Euler-Lagarange equation of J is,
\begin{equation}\label{1.12}
\frac{\partial h}{\partial r}-\frac{d}{dt}\frac{\partial h}{\partial \dot{r}}=0.
\end{equation}
From the Euler- Lagrange equation given in \ref{1.12} we have,
\begin{equation}
\begin{split}
(1-b^2)^{\frac{3}{2}}r+\lambda\left(\frac{r}{\sqrt{r^2+\dot{r}^2}}+b(\left( -\dot{r}\sin(\theta-t)+r\cos(\theta-t))\theta_{r} +\sin(\theta-t)\right)  \right)\\ -\lambda\frac{d}{dt}\left(\frac{\dot{r}}{\sqrt{r^2+\dot{r}^2}}  +b\cos(\theta-t)\right) =0.
\end{split}
\end{equation}
Along $\gamma_0(t)=(a\cos t, a\sin t)$, $r$ is constant. Hence, $\theta_r=0$.  Therefore, the above equation reduces to,
      \begin{equation}\label{1.1.0}
      \lambda\left[ 1+b\dot{\theta}\sin(\theta-t)\right] =-a(1-b^2)^{\frac{3}{2}}.
      \end{equation}
Hence,
 \begin{equation}\label{eq1.22}
 \lambda=-\frac{a(1-b^2)^{\frac{3}{2}}}{1+b\dot{\theta}\sin(\theta-t)}.
 \end{equation}
To prove Theorem \ref{thm1}, we require $\lambda$ to be a negative constant. This is achieved if, $\dot{\theta}\sin(\theta-t)$ is constant and greater than $-1$. Therefore, for the sake of simplicity we may assume $\dot{\theta}\sin(\theta-t)=0$. then clearly $\lambda<0$. Hence, either $\theta=$ constant along $\gamma_0$, or, $\theta-t=n\pi$, for all, $n\in \mathbb{Z}$. More generally, if we consider $\theta-t=c$, for some constant $c$, then also $\lambda<0$, as $b<1$.\\
   To prove Theorem \ref{thm1}, we assume  $\theta=t+c$ and Corollary \ref{cor2} will be dealt with the case,  $\theta\equiv$ constant.
\begin{proposition}\label{prop1.02}
 The circles
\begin{equation}\label{1.14}
\gamma_0=(a\cos t, a\sin t)
\end{equation} are normal as per Definition \ref{def1.001}.
\end{proposition}
\begin{proof}
Let us consider,
\begin{equation}
f(x^{1},x^{2},\dot{x}^1,\dot{x}^{2})=\frac{(1-b^2)^{\frac{3}{2}}}{2}(x^{1}\dot{x}^{2}-x^{2}\dot{x}^{1})
\end{equation}
and
\begin{equation}\label{1.1.17}
g(x^{1},x^{2},\dot{x}^{1},\dot{x}^{2})=\sqrt{(\dot{x}^{1})^2+(\dot{x}^{2})^2}+b(\cos (t+c) \dot{x}^1 +\sin (t+c) \dot{x}^2).
\end{equation}
Differentiating $g$ with respect to $x^1$, $x^2$, $\dot{x}^{1}$ and $\dot{x}^{2}$ respectively, we get,
\begin{equation}\label{1.1.19}
g_{x^{1}}=g_{x^{2}}=0, \quad g_{\dot{x}^{1}}=\frac{\dot{x}^{1}}{\sqrt{(\dot{x}^{1})^2+(\dot{x}^{2})^2}}+b\cos (t+c), \quad  g_{\dot{x}^{2}}=\frac{\dot{x}^{2}}{\sqrt{(\dot{x}^{1})^2+(\dot{x}^{2})^2}}+b\sin (t+c).
\end{equation}
Along the circle given by $\gamma_0$ we have,
\begin{equation}\label{1.1.21}
g_{\dot{x}^{1}}=-\sin t+ b\cos (t+c),~~g_{\dot{x}^{2}}=\cos t+ b\sin (t+c).
\end{equation}
Since,
\begin{equation}\label{1.1.22}
P_{j}=g_{x^{j}}-\frac{d}{dt}g_{\dot{x}^{j}},\quad \textnormal{for} ~~j=1,2,
\end{equation}
along $\gamma_0$ we have,
\begin{equation*}
P_{1}=\cos t+b\sin (t+c), \quad P_{2}=\sin t-b\cos (t+c).
\end{equation*}
Therefore, it follows that 
\begin{equation}\label{1.1.23}
(P_{1},P_{2})\neq 0,
\end{equation}
and hence the circles $\gamma _{0}$ are normal.
\end{proof}
\begin{proposition}\label{prop1.03}
If $\lambda <0$, then the Weierstrass function $E$ of the integral $J$ satisfies 
\begin{equation}\label{1.1.26}
E(x^{1},x^{2},\dot{x}^{1},\dot{x}^{2},u^{1},u^{2})\leq 0.
\end{equation}
Moreover, the equality holds if and only if $(u^{1},u^{2})=k(\dot{x}^{1},\dot{x}^{2})$, where $k$ is a positive constant.
\end{proposition}
\begin{proof}
In this case we obtain,
\begin{equation}\label{1.1.28}
\frac{\partial h}{\partial \dot{x}^{1}}=-\frac{(1-b^2)^{\frac{3}{2}}}{2}x^{2}+\lambda\left(  \frac{\dot{x}^{1}}{\sqrt{(\dot{x}^{1})^{2}+(\dot{x}^{2})^{2}}}+b\cos (t+c)\right) 
\end{equation}
and
\begin{equation}\label{1.1.29}
\frac{\partial h}{\partial \dot{x}^{2}}=\frac{(1-b^2)^{\frac{3}{2}}}{2}x^{1}+\lambda\left(  \frac{\dot{x}^{2}}{\sqrt{(\dot{x}^{1})^{2}+(\dot{x}^{2})^{2}}}+b\sin (t+c)\right).
\end{equation}
Using \ref{1.1.28} and \ref{1.1.29} in \ref{1.25} and simplifying we get,
\begin{equation}\label{1.1.33}
E(x,\dot{x},u)=\lambda\left( \sqrt{(u^{1})^{2}+(u^{2})^{2}}-\frac{\dot{x}^{1}u^{1}+\dot{x}^{2}u^{2}}{\sqrt{(\dot{x}^{1})^{2}+(\dot{x}^{2})^{2}}}\right) =\frac{\lambda}{\| \dot{x}\|}\left( \| u\| \| \dot{x}\| -(\dot{x}^{1}u^{1}+\dot{x}^{2}u^{2})\right).
\end{equation}
Now using  Cauchy-Schwarz  inequality,
\begin{equation}\label{1.1.34}
(\| u\| \| x\| -(\dot{x}^{1}u^{1}+\dot{x}^{2}u^{2}))\ge 0,
\end{equation}
and the fact that, $\lambda \le 0$, we obtain
\begin{equation}\label{1.1.35}
E(x,\dot{x},u)\le 0.
\end{equation}
It is evident that the equality holds if and only if $u=k\dot{x}$, for any constant $k$. Hence, we conclude the proof. 
\end{proof}
\begin{proposition}\label{prop1.04}
The Jacobi equation along the isoperimetric extremal circles $\gamma _{0}$ of the Randers metric $F=\sqrt{(dx^1)^2+(dx^2)^2}+ b(\cos \theta(t+c) dx^1 +\sin\theta(t+c) dx^2)$, is given by 
\begin{equation}\label{1.1.37}
\frac{d^{2}\omega}{dt^{2}}+\omega-\frac{\mu a^{2}}{\lambda}=0,
\end{equation} 
where $\mu$ is a constant.
\end{proposition}
\begin{proof}
For the circle $\gamma_0$ defined by \ref{1.14} we get,
\begin{equation}\label{1.1.38}
h=(1-b^2)^{\frac{3}{2}}(x^{1}\dot{x}^{2}-x^{2}\dot{x}^{1})+\lambda\left( \sqrt{(\dot{x}^{1})^2+(\dot{x}^{2})^2}+b\cos(t+c) \dot{x}^1 + \sin (t+c) \dot{x}^2\right). 
\end{equation}
From \ref{1.1.17} we have,
\begin{equation}\label{1.40}
g_{\dot{x}^{1}\dot{x}^{1}}=\frac{\cos ^{2}t}{a}, \quad g_{\dot{x}^{2}\dot{x}^{2}}=\frac{\sin ^{2}t}{a}, \quad g_{x^{1}\dot{x}^{2}}=g_{\dot{x}^{1}x^{2}}=0.
\end{equation}
Let us define 
\begin{equation}\label{1.52}
U :=g_{x^{1}\dot{x}^{2}}- g_{\dot{x}^{1}x^{2}}-g_{\dot{x}^{1}\dot{x}^{1}}\frac{d}{dt}\left( \frac{\dot{x}^{1}}{\dot{x}^{2}}\right). 
\end{equation}
Then along the circles $\gamma_0$, we get $U=\frac{1}{a}$. Now differentiating $h$ with respect to $\dot{x}^1$ twice,
\begin{equation}\label{1.45}
h_{\dot{x}^{1}\dot{x}^{1}}=f_{\dot{x}^1\dot{x}^{1}}+\lambda g_{\dot{x}^{1}\dot{x}^{1}}.
\end{equation}
Let us define $h_{1}:=\frac{h_{\dot{x}^1\dot{x}^{1}}}{(\dot{x}^{2})^{2}}$. Then along $\gamma_0$, $h_1=\frac{\lambda}{a^{3}}$ and $\frac{dh_{1}}{dt}=0$. Now suppose
\begin{equation*}
K :=h_{x^{1}\dot{x}^{1}}-\dot{x}^{2} \ddot{x}^{2}h_{1}.
\end{equation*} 
Then along the curve $\gamma_0$, we obtain, $K=\frac{\lambda \sin t\cos t}{a}$ and $\frac{dK}{dt}=\frac{\lambda \cos 2t}{a}$.\\
Differentiating $f$ and $g$ with respect to $x^{1}$ twice we get $g_{x^{1}x^{1}}=0$, $f_{x^{1}x^{1}}=0$. Hence, $h_{x^{1}x^{1}}=f_{x^{1}x^{1}}+\lambda g_{x^{1}x^{2}} =0$. Let us define
\begin{equation*}
h_{2}:=\frac{1}{(\dot{x}^{2})^{2}}\left( h_{x^{1}x^{1}}-(\ddot{x}^{2})^{2}h_{1}-\frac{dK}{dt}\right) .
\end{equation*} Then for the circle $\gamma_0$, we have $h_2= -\frac{\lambda}{a^{3}}$.\\
Hence, the Jacobi equation along $\gamma_0$ is of the form,
\begin{equation}\label{1.55}
\Psi+\mu U =0,
\end{equation}
where, 
\begin{equation}\label{1.56}
\Psi (\omega)=h_{2}\omega-\frac{d}{dt}(h_{1}\omega ').
\end{equation}
Here we see that $h_{1}'=0$. Hence from the above values of $h_1$ and $h_2$ we have 
\begin{equation}\label{1.57}
\frac{d^{2}\omega}{dt^{2}}+\omega-\frac{\mu a^{2}}{\lambda}=0.
\end{equation}
\end{proof}\\
Along $\gamma_0 (t)$, a point $\gamma_0 (c)$ will be conjugate point of $\gamma_0 (a)$, if there exist a non-zero solution $\omega$ of the Jacobi equation \ref{1.1.37} along $\gamma_0 (t)$ such that 
\begin{equation}\label{1.58}
\omega (a) = \omega (c) =0
\end{equation}
and
\begin{equation}\label{1.59}
\int\limits_{a}^{c} U \omega dt = 0,
\end{equation}
where $U$ is given by \ref{1.52}, \cite{OB,MH}.
\begin{proposition}\label{prop1.05}
 There is no point which is conjugate to $\gamma (t_0)$ along the isoperimetric extremal circle $\gamma _{0}$.
\end{proposition}
\begin{proof}
Solving the differential equation \ref{1.57} we get 
\begin{equation}
\omega (t) = c_1\cos t+c_2\sin t+\frac{\mu a^2}{\lambda}.
\end{equation}
Suppose that $t_0,t_1\in[0,2\pi)$ are such that $\gamma(t_1)$ is conjugate to $\gamma(t_0)$. Then $\omega$ satisfy $\ref{1.58}$ and $\ref{1.59}$. Hence, we have,
\begin{equation}
c_1\cos t_0+c_2\sin t_0+\frac{\mu a^2}{\lambda}=0,
\end{equation}
\begin{equation}
c_1\cos t_1+c_2\sin t_1+\frac{\mu a^2}{\lambda}=0,
\end{equation}
\begin{equation}
Uc_1\int\limits_{t_0}^{t_1}\cos t dt+Uc_2\int\limits_{t_0}^{t_1}\sin t dt+U\mu \int\limits_{t_0}^{t_1} \frac{a^2}{\lambda}dt=0.
\end{equation}
The above system of equation has non-zero solution if and only if the determinant of the matrix $P$ is zero, where
\begin{equation*}
P=\begin{pmatrix}
\cos t_0 & \sin t_0 & \frac{a^2}{\lambda} \\ \cos t_1 & \sin t_1 & \frac{a^2}{\lambda} \\ U\int\limits_{t_0}^{t_1}\cos t dt & U\int\limits_{t_0}^{t_1}\sin t dt & U \int\limits_{t_0}^{t_1} \frac{a^2}{\lambda}dt
\end{pmatrix}.
\end{equation*}
Then we get,
\begin{equation}
\det(P)=\frac{a^2}{\lambda}U[(t_1-t_0)\sin (t_1- t_0)+2\cos(t_1-t_0)-2].
\end{equation}
It can be easily seen that $\det(P)=0$ if and only if $t_1=t_0$. This implies that $\gamma_0$ does not have any conjugate points.
\end{proof}

\begin{proposition}\label{prop1.06}
Along the isoperimetric extremal circle $\gamma_{0}$, the following inequality holds:
\begin{equation}\label{1.1.61}
\sum\limits_{i,j =1}^{2}h_{\dot{x}^{i}\dot{x}^{j}}y^{i}y^{j}\leq 0,
\end{equation}
and the equality holds if and only if $(y^{1},y^{2})=k(\dot{x}^1,\dot{x}^2)$.
\end{proposition}
\begin{proof}
Differentiating \ref{1.1.28} with respect to $\dot{x}^1$ and \ref{1.1.29} with respect to $\dot{x}^2$ we yield,
\begin{equation}
h_{\dot{x}^1\dot{x}^1}= \frac{\lambda (\dot{x}^2)^2 }{((\dot{x}^1)^2+(\dot{x}^2)^2)^{\frac{3}{2}}}, ~ h_{\dot{x}^1\dot{x}^2}=-\frac{ \dot{x}^1\dot{x}^2 }{((\dot{x}^1)^2+(\dot{x}^2)^2)^{\frac{3}{2}}} , ~h_{\dot{x}^2\dot{x}^2}= \frac{\lambda (\dot{x}^1)^2 }{((\dot{x}^1)^2+(\dot{x}^2)^2)^{\frac{3}{2}}}.
\end{equation}
Therefore,
 \begin{equation}\label{1.1.63}
\sum\limits_{i,j =1}^{2}h_{\dot{x}^{i}\dot{x}^{j}}y^{i}y^{j} = \frac{\lambda}{((\dot{x}^1)^2+(\dot{x}^2)^2)^{\frac{3}{2}}}(\dot{x}^2y^1-\dot{x}^1y^2)^2 = \frac{\lambda}{a}(y^{1}\cos t+y^{2}\sin t)^{2}.
\end{equation}
As $\lambda$ is negative and $a>0$, the inequality \ref{1.1.61} follows. Clearly the equality holds if and only if $(y^1,y^2)=k(\dot{x}^1,\dot{x}^2)$.
\end{proof}
\vspace{0.4cm}

\textbf{Proof of Theorem \ref{thm1}:}\\
In propositions \ref{prop1.01}-\ref{prop1.06}, we have shown that $\gamma_0$ satisfies all the conditions of Theorem \ref{thm1.01}. Hence, $\gamma_0$ is a solution of the isoperimetric problem with respect to the Busemann-Hausdorff measure. Similarly, it also follows that $\gamma_0$ is a solution of the isoperimetric problem with respect to the Holmes-Thompson, maximum, minimum volume forms as these volume forms are just the scalar multiples of Busemann-Hausdorff volume form.

\vspace{0.4cm}
\textbf{Proof of Corollary \ref{cor1}:}\\ Choosing $\theta=0$ in \ref{eq1.22}, we obtain $\lambda=-\frac{a}{(1-b^2)^{\frac{3}{2}}}<0$ and $F=\sqrt{(dx^1)^2+(dx^2)^2}+ bdx^1, ~b>0$. Therefore, following the same steps of Theorem \ref{thm1}, we complete the proof.
\vspace{0.4cm}

\textbf{Proof of Corollary \ref{cor2}:}\\
Since $\tau:\mathbb{R}^2\to \mathbb{R}$ is a smooth function, we have $d\tau=\frac{\partial \tau}{\partial x^1}dx^1+\frac{\partial \tau}{\partial x^2}dx^2$. Therefore, $\|d\tau\|^2=\left( \frac{\partial \tau}{\partial x^1}\right)^2 +\left( \frac{\partial \tau}{\partial x^2}\right)^2=v $, for some constant $v$. Writing $\frac{\partial \tau}{\partial x^1}=p$ and $\frac{\partial \tau}{\partial x^2}=q$, we obtain the following partial differential equation:
\begin{equation}
v=p^2+q^2.
\end{equation}
Using Charpit's method (p.69, \cite{INS}) we obtain that $p=b_1$ and $q=b_2$, where $b_1$ and $b_2$ are some real constants with $b_1^2+b_2^2<1$. Hence, $d\tau=b_1dx^1+b_2dx^2$. Therefore, we can write, $b_1=b\cos c$ and $b_2=b\sin c$, for some constants $b$ and $c$. Therefore, choosing $\theta\equiv c$, in Theorem \ref{thm1} we conclude the proof.

\end{document}